\documentclass{article}
\usepackage{graphicx} 
\usepackage{amsmath,amssymb}
\usepackage{amsthm}
\usepackage[abbrev]{amsrefs}
\usepackage{latexsym}
\usepackage{mathtools}
\usepackage{stmaryrd}
\numberwithin{equation}{section}
\allowdisplaybreaks
\newtheorem{thm}{Theorem}[section]

\newtheorem{lem}[thm]{Lemma}

\theoremstyle{definition}
\newtheorem{dfn}[thm]{Definition}
\theoremstyle{remark}
\newtheorem*{rem*}{Remark}

\title{Congruences for the coefficients of MacMahon-like $q$-series}
\author{Yuta Motomura, Takehiro Suda}
\date{June 2026}

\begin{document}
	
	\maketitle
	
	\begin{abstract}
		In this paper, we study the congruences for the coefficients of MacMahon-like series introduced by Bringmann--Craig--van~Ittersum--Pandey~\cite{BCIP}. We prove all the congruences conjectured in \cite{BCIP}. In addition, we establish several new congruence relations of a similar type.
	\end{abstract}
	
	\section{Introduction}
	MacMahon \cite{MacMahon} introduced for $a\in\mathbb{N}$ the $q$-series defined by
	$$\mathcal{U}_a(q)~\coloneqq~\sum_{1\leq n_1< \cdots<n_a}{\frac{q^{n_1+\cdots+n_a}}{(1-q^{n_1})^2\cdots{(1-q^{n_a})^2}}}.$$
	We consider $\mathcal{U}_a(q)$ as generalized generating functions of \textit{sum--of--divisors functions} because we have
	$$\mathcal{U}_a(q)~=~\sum_{\substack{1\leq n_1<\cdots<n_a \\ 1\leq m_1,\cdots,m_a}}m_1m_2\cdots m_aq^{n_1m_1+\cdots+n_am_a}.$$
	These functions have been studied from various perspectives. Andrews--Rose \cite{AR} showed that $\mathcal{U}_a$ is a quasimodular form. More recently, Amdeberhan--Andrews--Tauraso~\cite{AAT} and Amdeberhan--Ono--Singh \cite{AOS} showed many congruences for the Fourier coefficients of $\mathcal{U}_a$. Furthermore, Bringmann--Craig--van~Ittersum--Pandey~\cite{BCIP} introduced MacMahon-like functions, motivated by the work of Bachmann~\cite{BK}, as generalizations of MacMahon series, defined by
	$$\mathcal{A}_{a,k,r}(q)~\coloneqq~\sum_{1\leq n}{c_{a,k,r}(n)q^n}
	~\coloneqq~\sum_{1\leq n_1<\cdots<n_a}
	\frac{q^{r(n_1+\cdots+n_a)}}{(1-q^{n_1})^k\cdots{(1-q^{n_a})^k}},$$
	$$\mathcal{B}_{a,k,r,s}(q)~\coloneqq~\sum_{1\leq n}{d_{a,k,r,s}(n)q^n}
	~\coloneqq~\sum_{1\leq n_1<\cdots<n_a}
	\frac{q^{r(n_1^2+\cdots+n_a^2)+s(n_1+\cdots+n_a)}}{(1-q^{n_1})^k\cdots{(1-q^{n_a})^k}}.$$
	They showed $\mathcal{A}_{a,2r,r}$ is a quasimodular form and satisfies infinitely many congruences of the coefficients.
	
	\begin{thm}[Bringmann--Craig--van~Ittersum--Pandey\cite{BCIP}]\label{thm:bcip}
		Let $a,r,m\in\mathbb{N}$. There are infinitely many non-nested arithmetic progressions $An+B$ such that
		$$c_{a,2r,r}(An+B)\equiv0\pmod{m}.$$
	\end{thm}
	In Theorem 1.1, they showed that the coefficients of $\mathcal{A}_{a,2r,r}$ satisfy many congruences, while several explicit congruences for certain coefficients $\mathcal{A}_{a,k,r}$ and $\mathcal{B}_{a,k,r,s}$ were left as conjectures. In this paper, we solve the conjectures in \cite{BCIP} and derive numerous congruences of the coefficients of $\mathcal{A}_{a,k,r}$ and $\mathcal{B}_{a,k,r,s}$ by elementary calculation. 
	
	The following theorem, originally conjectured as Conjecture 5.1 in \cite{BCIP}, is established below:
	\begin{thm}\label{thm:conj5.1}
		{\ }
		\begin{enumerate}
			\item Let $n\in\mathbb{Z}_{\ge0}$ and $a, k, r\in\mathbb{N}$. For any prime $p$, we have
			$$c_{a,k,r}(p^{\alpha+1}n+p^{\alpha}\beta)\equiv 0\pmod{p^{v_p(k)-\alpha}}$$
			for any $0 \le\alpha\le v_p(\gcd(k, r))-1$ and $1 \le \beta \le p-1$. Here, $v_p(n)$ denotes the $p$-adic valuation of $n$.
			\item Let $n\in\mathbb{Z}_{\ge0}$ and $a, k, r, s\in\mathbb{N}$. For any prime $p$, we have $$d_{a,k,r,s}(p^{\alpha+1}n+p^{\alpha}\beta)\equiv 0\pmod{p^{v_p(k)-\alpha}}$$
			for any $0 \le\alpha\le v_p(\gcd(k, r,s))-1$ and $1 \le \beta \le p-1$.
			
			\item Let $a, k, n  \in \mathbb{N}$ and $r, s \in 2\mathbb{Z}_{\ge0}+1$. Then we have
			$$d_{a,k,r,s}(2n-1) \equiv 0\pmod{2^{v_2(k)}}.$$
		\end{enumerate}
		
	\end{thm}
	We call these congruences \emph{Hecke-type congruences} because similar congruences arise for Hecke eigenforms. For example, the Ramanujan tau function $\tau(n)$, which gives the Fourier coefficients of the modular discriminant $\Delta$, satisfies $\tau(5^\alpha n)\equiv0\pmod{5^\alpha}$ that is called Hecke congruence.
	
	The following congruences were also conjectured in \cite{BCIP}. Statements (1) and (2) coincide with conjectures appearing there, while statement (3) provides a generalization of the conjectured form.
	
	\begin{thm}\label{thm:specific_cong}
		{\ }
		\begin{enumerate}
			\item Let $n\in\mathbb{N}$. Then we have
			$$c_{1,3,1}(8n-4)\equiv0\pmod{7}.$$
			\item For every pair of positive integers $a$ and $n$, we have
			$$c_{3a,4,2}(3n-1)\equiv c_{3a-1,4,2}(3n-1)\equiv 0\pmod{3}.$$
			\item Let $n\in \mathbb{Z}$ and p be an odd prime. Then we have$$c_{1,2p-1,p-1}(p^2n+1)\equiv c_{1,2p-1,p}(p^2n+1) \equiv0\pmod{p}.$$
		\end{enumerate}
	\end{thm}

	\begin{rem*}
		The conjecture $$c_{2,4,2}(37n)\equiv 0\pmod{19}$$
		in the earlier version of \cite[Conjecture 5.1]{BCIP} does not hold in general $n$; indeed, for $n = 222$ we have $c_{2,4,2}(37\cdot222)\equiv 1\pmod{19}$.
	\end{rem*}

	Finally, we establish an analogue of Theorem \ref{thm:bcip} in the case where $a=1$ and $k\in\mathbb{N}$.
	\begin{thm}\label{thm;arithmetic}
		Let $k, r, m \in \mathbb{N}$. There are infinitely many non-nested arithmetic progressions
		$An + B$ such that $$c_{1,k,r}(An+B) \equiv 0 \pmod{m}.$$
	\end{thm}

	\section*{Acknowledgements}
	The authors would like to express their sincere gratitude to their supervisor, Professor Yasuo Ohno, for carefully reading and providing helpful guidance. Additionally, the authors are deeply grateful to Jan-Willem van~Ittersum for his kind encouragement and for sharing the code he developed during the intermediate stages of this research.
	
	\section{Proofs of Theorem \ref{thm:conj5.1}}
	\subsection{Proof of Theorem \ref{thm:conj5.1}(1) and (2)}
	As noted in \cite[Corollary 1.4]{BCIP}, the following expressions hold.
	\begin{lem}\label{lem:express}
		Let $a,k,r,s\in\mathbb{N},$ then we have
		$$c_{a,k,r}(n)=\sum_{\substack{1\le n_1< n_2<\dots< n_a\\r\le t_1,t_2,\dots,t_a\\ \sum_{i=1}^{a}n_it_i=n}}\prod_{i=1}^{a}\binom{k-1+t_i-r}{k-1},$$
		$$d_{a,k,r,s}(n)= \sum_{\substack{1\leq n_1< n_2<\cdots< n_a\\s\leq t_1,t_2,\cdots,t_a\\r(n_1^2+\cdots+n_a^2)+\sum_{i=1}^{a}n_it_i=n}\\}\prod_{i=1}^{a}\binom{k-1+t_i-s}{k-1}.$$
	\end{lem}
	
	\begin{proof}
		Expanding each term of $\mathcal{A}_{a,k,r}$ as a geometric series, we can calculate as follows.
		Here, $[q^n]f(q)$ denotes the coefficient of $q^n$ in $f(q)$.
		\begin{align*}
			c_{a,k,r}(n)&=[q^n]\sum_{1\le n_1< n_2<\dots< n_a}\prod_{i=1}^{a}\frac{q^{rn_i}}{(1-q^{n_i})^k}\\
			&=\sum_{1\le n_1< n_2<\dots< n_a}\sum_{\substack{0\le s_1,s_2,\dots,s_a\\ \sum_{i=1}^{a}s_i=n}}\prod_{i=1}^{a}[q^{s_i}]\frac{q^{rn_i}}{(1-q^{n_i})^k}\\
			&=\sum_{1\le n_1< n_2<\dots< n_a}\sum_{\substack{0\le s_1,s_2,\dots,s_a\\n_i\mid s_i, s_i\ge rn_i\\ \sum_{i=1}^{a}s_i=n}}\prod_{i=1}^{a}[q^{s_i-rn_i}]\frac{1}{(1-q^{n_i})^k}\\
			&=\sum_{1\le n_1< n_2<\dots< n_a}\sum_{\substack{0\le t_1,t_2,\dots,t_a\\ \sum_{i=1}^{a}n_i(r+t_i)=n}}\prod_{i=1}^{a}[q^{t_i}]\frac{1}{(1-q)^k}\\
			&=\sum_{1\le n_1< n_2<\dots< n_a}\sum_{\substack{0\le t_1,t_2,\dots,t_a\\ \sum_{i=1}^{a}n_i(r+t_i)=n}}\prod_{i=1}^{a}\binom{k-1+t_i}{k-1}\\
			&=\sum_{\substack{1\le n_1< n_2<\dots< n_a\\r\le t_1,t_2,\dots,t_a\\ \sum_{i=1}^{a}n_it_i=n}}\prod_{i=1}^{a}\binom{k-1+t_i-r}{k-1}.
		\end{align*}
		We note that since the range of the summation and denominator of the summand for $\mathcal{B}_{a,k,r,s}$ are the same, almost the same argument goes through for $d_{a,k,r,s}.$
	\end{proof}
	
	The second lemma we use is a classical result due to Kummer.
	
	\begin{lem}\label{lem:Kummer}\cite{Kummer}
		Let $m, n\in\mathbb{Z}_{\ge0}$ and $p$ be prime. If $m\leq n$, then $v_p\left(\binom{n}{m}\right)$ equals the number of carries that occur in the usual base-$p$ addition of $m$ and $n-m$.
	\end{lem}

	We now proceed to prove Theorem \ref{thm:conj5.1}(1) and (2).
	\begin{proof}[Proof of Theorem \ref{thm:conj5.1}(1) and (2)]
		Let $k=p^eu,\ r=p^fw$, where $e= v_p(k),\ f=v_p(r)$. 
		Under these assumptions, the constraint on $\alpha$ can be written as $0\le\alpha\le\min(e,f)-1$.
		By Lemma \ref{lem:express}, we have        
		$$\begin{aligned}
			v_p\left(c_{a,p^eu,p^fw}(p^{\alpha+1}n+p^{\alpha}\beta)\right)&\ge\min_{\substack{1\le n_1< n_2<\dots< n_a\\p^{f}w< t_1,t_2,\dots,t_a\\ \sum_{i=1}^{a}n_it_i=p^{\alpha}(pn+\beta)}}\sum_{i=1}^{a}v_p\left(\binom{p^eu-1+t_i-p^fw}{p^eu-1}\right).
		\end{aligned}$$
		
		When $(n_1,n_2,\dots,n_a)$ and $(t_1,t_2,\dots,t_a)$ satisfy the minimization condition, by the property of $\sum n_it_i=p^{\alpha}(pn+\beta)$, there exists an index $j$ such that $v_p(t_j)\le\alpha$. For this $j$, the base-$p$ expansion of $p^eu-1$ has its first $e$ digits all equal to $p-1$, and the base-$p$ expansion of $t_j-p^fw$ has its first $v_p(t_j)$-th digits equal to $0$ and the coefficient of $p^{v_p(t_j)}$ is nonzero. Then, at least $e-v_p(t_j)$ carries occur in the addition of $p^eu-1$ and $t_j-p^fw$ written in the base $p$. Hence, by Lemma \ref{lem:Kummer}, the following holds:
		
		$$v_p\left(\binom{p^eu-1+t_j-p^fw}{p^eu-1}\right)>e-v_p(t_j)\ge e-\alpha.$$
		Therefore, we can conclude that
		$$\begin{aligned}
			v_p\left(c_{a,p^eu,p^fw}(p^{\alpha+1}n+p^{\alpha}\beta)\right)&\ge e-\alpha.
		\end{aligned}$$
		A similar discussion holds for $d_{a,k,r,s}$.
	\end{proof}
	\subsection{Proof of Theorem \ref{thm:conj5.1}(3)}
	Theorem \ref{thm:conj5.1}(3) can be proved in a similar manner.
	\begin{proof}[Proof of Theorem \ref{thm:conj5.1}(3)]
		By Lemma \ref{lem:express}, we have
		$$\begin{aligned}
			v_2\left(d_{a,k,r,s}(2n-1)\right)&\ge\min_{\substack{1\le n_1< n_2<\dots< n_a\\s< t_1,t_2,\dots,t_a\\ r\sum_{i=1}^a n_i^2+\sum_{i=1}^{a}n_it_i=2n-1}}\sum_{i=1}^{a}v_2\left(\binom{k-1+t_i-s}{k-1}\right).
		\end{aligned}$$
		Here, from the property of $r\sum n_i^2+\sum n_it_i=2n-1$, there exists an index $j$ such that $t_j$ is even. For this $j$, it follows by Lemma \ref{lem:Kummer} that the following holds:
		
		$$v_2\left(\binom{k-1+t_j-s}{k-1}\right)>v_2(k)-1.$$
		
	\end{proof}
	\section{Proofs of Theorem \ref{thm:specific_cong} and \ref{thm;arithmetic}}
	\subsection{Proof of Theorem \ref{thm:specific_cong}(1)}
	\begin{proof}[Proof of Theorem \ref{thm:specific_cong}(1)]
		Using Lemma \ref{lem:express}, we obtain the following:
		\begin{align*}
			c_{1,3,1}(8n-4)&=\sum_{\substack{1\le n_1\\1\le t_1\\n_1t_1=8n-4}}\binom{t_1+1}{2}\\
			&=\sum_{t|8n-4}\dfrac{t(t+1)}{2}\\
			&=\sum_{t| 2n-1}\left(\dfrac{t(t+1)}{2}+\dfrac{2t(2t+1)}{2}+\dfrac{4t(4t+1)}{2}\right)\\
			&=\sum_{t|2n-1}\dfrac{7t(3t+1)}{2}\\
			&\equiv 0\pmod{7}.
		\end{align*}
	\end{proof}

	\subsection{Proof of Theorem \ref{thm:specific_cong}(2)}
	Before proving the Theorem \ref{thm:specific_cong}(2), we introduce two functions and two lemmas.
	\begin{dfn}
		For $n\ge t$, define $$h(n,t)\coloneqq\frac{(2n+1)(n+t)!}{(2t+1)!(n-t)!},$$ and for $n\in\mathbb{Z}_{\ge 0},$ define
		$$W_{a}(n)\coloneqq\sum_{t_1+t_2=2a}\sum_{\substack{n_1\ge t_1\\n_2\ge t_2\\\frac{n_1^2+n_1+n_2^2+n_2}{2}=n}}(-1)^{n_1+n_2}h(n_1,t_1)h(n_2,t_2).$$
	\end{dfn}
	
	\begin{lem}\label{lem:v3h}
		For all $n\equiv 1\pmod{3}$ and $t\le n$, we have
		$$v_{3}(h(n,t))\ge\begin{cases}
			0&(t\equiv 1\pmod{3})\\
			1&(t\not \equiv 1\pmod{3})\\
		\end{cases}.$$
	\end{lem}
	\begin{proof}Due to the identity   $v_3(h(n,t))=v_3(2n+1)-v_3(2t+1)+v_3\left(\binom{n+t}{2t}\right),$ the case $v_3(2n+1)\ge v_3(2t+1)$ is trivial. 
		
		For the remaining case, by the condition that $v_3(n-t)=v_3((2n+1)-(2t+1))=v_3(2n+1)$ and using Lemma \ref{lem:Kummer}, an inequality $v_3(\binom{n+t}{2t})\ge v_3(2t+1)-v_3(2n+1)$ is fulfilled, and hence $v_3(h(n,t))\ge 0.$ 
	\end{proof}
	
	\begin{lem}\label{lem:W3a}
		For all $a,n\in\mathbb{N},$ we have that
		$$W_{3a}(3n-1)\equiv W_{3a-1}(3n-1)\equiv 0\pmod{3}$$
	\end{lem}
	\begin{proof}
		It is easy to see that the condition $\frac{n_1^2+n_1+n_2^2+n_2}{2}\equiv 2\pmod{3}$ implies $n_1\equiv n_2\equiv 1\pmod{3}.$ We can also see that $t_1\not\equiv1\pmod{3}$ or $t_2\not\equiv 1\pmod{3}$ when $t_1+t_2\not\equiv1\pmod{3}.$
		
		Therefore, using Lemma \ref{lem:v3h}, all of the summands of $W_{3a}(3n-1)$ and $W_{3a-1}(3n-1)$ are congruent to $0\pmod{3}.$
	\end{proof}
	Using the above lemmas, we prove Theorem \ref{thm:specific_cong}(2).
	\begin{proof}[Proof of Theorem \ref{thm:specific_cong}(2)]
		Let $$F(x,q)\coloneqq 2\sum_{n=0}^{\infty}T_{2n+1}\left(\frac{x}{2}\right)q^{n^2+n},$$
		where $T_n(x)$ is the Chebyshev polynomial of the first kind of degree $n$.
		In \cite{AR}, it is shown that
		\begin{align*}
			F(x,q)&=x(q^2;q^2)_{\infty}^{3}\prod_{m=1}^{\infty}\left(1+x^2\frac{q^{2m}}{(1-q^{2m})^2}\right)\\&=(q^2;q^2)_{\infty}^{3}\sum_{a=0}^{\infty}\mathcal{A}_{a,2,1}(q^2)x^{2a+1}.
		\end{align*}
		Then, the following identity holds:
		\begin{align}
			F(x,q)F(-x,q)&=-x^2(q^2;q^2)_{\infty}^{6}\prod_{m=1}^{\infty}\left(1-x^4\frac{q^{4m}}{(1-q^{2m})^4}\right)\notag\\  &=(q^2;q^2)_{\infty}^{6}\sum_{a=0}^{\infty}(-1)^{a+1}\mathcal{A}_{a,4,2}(q^2)x^{4a+2}.\label{eq:equation_a42}
		\end{align}
		In \cite{AR}, it is also shown that
		$$F(x,q)=\sum_{t=0}^{\infty}x^{2t+1}f_t(q),$$
		where 
		$$f_t(q)\coloneqq\dfrac{(-1)^t}{(2t+1)!}\sum_{n=t}^{\infty}(-1)^n(2n+1)\dfrac{(n+t)!}{(n-t)!}q^{n^2+n}.$$
		Comparing the coefficients of $x^{4a+2}$ in (\ref{eq:equation_a42}), we can get
		\begin{align}\mathcal{A}_{a,4,2}(q^2)=(-1)^{a}(q^2;q^2)_{\infty}^{-6}\sum_{t_1+t_2=2a}f_{t_1}(q)f_{t_2}(q).\label{eq:Aa42}
		\end{align}
		We can rewrite (\ref{eq:Aa42}) as follows:
		$$\mathcal{A}_{a,4,2}(q)=(q;q)_{\infty}^{-6}\sum_{n=0}^{\infty}W_a(n)q^n.$$
		The fact that the coefficient of $q^n$ in $(q;q)_{\infty}^{-6}$ is congruent to $0$ modulo $3$ when $n\not\equiv0\pmod{3}$ and Lemma \ref{lem:W3a} induce the congruences of statement in Theorem \ref{thm:specific_cong}(2).
	\end{proof}

	\subsection{Proof of Theorem \ref{thm:specific_cong}(3)}
.
	
	\begin{proof}[Proof of Theorem \ref{thm:specific_cong}(3)]
		By Lemma \ref{lem:express}, we have $$c_{1,2p-1,r}(n)=\sum_{\substack{t\mid n\\0<t}}\binom{t-r+2p-2}{2p-2}.$$By Lemma \ref{lem:Kummer}, $\binom{t-r+2p-2}{2p-2}$ is a multiple of $p$ if and only if~ $$0\leq\exists s\leq p-2~~ \mathrm{s.t.}~~ t-r\equiv sp~~\mathrm{or}~~sp+1\pmod{p^2}.$$
		Replace $r$ with $p-1$ and $n$ with $p^2n+1$. Since $t-r$ is not congruent to $sp$ modulo $p^2$, we have $t\equiv sp-1~(1\leq s\leq p-1)$. Furthermore, $sp+1$ is an inverse element of $(p-s)p+1$ in $\mathbb{Z}/p^2\mathbb{Z}$. Thus it suffices to show that $\binom{(s+1)p-2}{2p-2}+\binom{(p-s+1)p-2}{2p-2}\equiv 0 \pmod{p}.$
		$$\begin{aligned}
			&\binom{(s+1)p-2}{2p-2}+\binom{(p-s+1)p-2}{2p-2}\\
			&=\frac{((s+1)p-2)\cdots((s-1)p+1)+((p-s+1)p-2)\cdots((p-s-1)p+1)}{(2p-2)!}\\
			&\equiv\frac{((s+1)p-2)\cdots((s-1)p+1)+((s-1)p+2)\cdots((s+1)p-1)}{(2p-2)!} \pmod{p} \end{aligned}$$
		By calculating the numerator, we have
		$$\begin{aligned}
			&\left(\sum_{i=1}^{2p-2}\frac{(2p-2)!}{i}\right)p+\left(\sum_{i=2}^{2p-1}\frac{(2p-1)!}{i}\right)p+(2p-2)!+(2p-1)!\equiv0\pmod{p^2}.
		\end{aligned}$$
		For $r=p$, $t-r$ is not congruent to $sp$ modulo $p$, so $t\equiv sp+1(1\leq s\leq p-1)$. Thus, almost the same argument goes through for $c_{1,2p-1,p}(p^2n+1)$.
	\end{proof}
	
	\begin{rem*}
		In [4], it is conjectured that $c_{1,5,2}(9n+1)\equiv 0\pmod{3}.$ This congruence corresponds to Theorem \ref{thm:specific_cong}(3) in the case $p=3$.
	\end{rem*}
	\subsection{Proof of Theorem \ref{thm;arithmetic}}

	\begin{proof}[Proof of Theorem \ref{thm;arithmetic}]
		By Dirichlet's theorem on arithmetic progressions, we can take a prime sequence $(p_i)_{i=1}^{\infty}$ such that $p_i\equiv 1\pmod{m(k-1)!}$ for all $i\in\mathbb{N}$.
		We will show that $\{(p_i^{m(k-1)!}n+p_i^{m(k-1)!-1})\mid i\in \mathbb{N}\}$ constitutes the desired family.    
		Let $$f(x)=\sum_{j=0}^{k-1}s_jx^j\coloneqq (x+k-r-1)(x+k-r-2)\cdots(x-r+1)\in \mathbb{Z}[x].$$ 
		Then, we can calculate 
		\begin{align*}
			(k-1)!c_{1,k,r}\left(p_i^{m(k-1)!}n+p_i^{m(k-1)!-1}\right)&=\sum_{t\mid p_i^{m(k-1)!}n+p_i^{m(k-1)!-1}}f(t)\\
			&=\sum_{t\mid p_in+1}\sum_{d=0}^{m(k-1)!-1}\sum_{j=0}^{k-1}s_j(tp_i^d)^j\\
			&=\sum_{j=0}^{k-1}\sum_{t\mid p_in+1}s_jt^j\sum_{d=0}^{m(k-1)!-1}p_i^{dj}\\
			&\equiv\sum_{j=0}^{k-1}\sum_{t\mid p_in+1}s_jt^j\sum_{d=0}^{m(k-1)!-1}1\\
			&\equiv 0\pmod{m(k-1)!}.
		\end{align*}   
		Hence, $$c_{1,k,r}\left(p_i^{m(k-1)!}n+p_i^{m(k-1)!-1}\right)\equiv 0\pmod{m}$$
		holds.
	\end{proof}

\end{document}